\documentclass[12pt]{article}
\usepackage[left=2.5cm,right=2.5cm,top=2.5cm,bottom=2.5cm,a4paper]{geometry}

\usepackage[a4paper]{geometry}
\usepackage{graphicx}
\usepackage{microtype}
\usepackage{siunitx}
\usepackage{booktabs}
\usepackage{cleveref}
\usepackage{graphics}
\usepackage{graphicx}
\usepackage{epsfig}
\usepackage{amsmath,amsfonts,amssymb,amsthm}
\usepackage{listings}
\usepackage{paralist}
\usepackage{sectsty}
\numberwithin{equation}{section}
\date{}
\subsectionfont{\normalfont}

\newtheorem{theorem}{Theorem}[section]
\newtheorem{definition}[theorem]{Definition}
\newtheorem{lemma}[theorem]{Lemma}
\newtheorem{corollary}[theorem]{Corollary}

\newtheorem{example}[theorem]{Example}
\newtheorem{proposition}[theorem]{Proposition}

\newtheorem{open problem}[theorem]{Open problem}

\providecommand{\keywords}[1]{\textbf{\textit{Key words---}} #1}

\begin{document}

\markboth{Kyunghwan Song}
{The Frobenius problem for four numerical semigroups}

\title{The Frobenius problem for four numerical semigroups}

\author{Kyunghwan Song\\
Department of mathematics, Korea University, Seoul 02841, Republic of Korea\\
heroesof@korea.ac.kr}

\maketitle

\begin{abstract}
The greatest integer that does not belong to a numerical semigroup $S$ is called the Frobenius number of $S$ and finding the Frobenius number is called the Frobenius problem.  In this paper,  we introduce the Frobenius problem for numerical semigroups generated by Thabit number base b and Thabit number of the second kind base b which are motivated by the Frobenius problem for Thabit numerical semigroups. Also, we introduce the Frobenius problem for numerical semigroups generated by Cunningham number and Fermat number base $b$.
\end{abstract}



\section{Introduction}\label{sec_Introduction}
Let $\mathbb{N}$ be the set of nonnegative integers.
\begin{definition} \cite{Rosales2015,Rosales2009}
A {\em numerical semigroup} is a subset $S$ of $\mathbb{N}$ that is closed under addition, contains $0$ and $\mathbb{N}\texttt{\char`\\}S$ is finite. Given a nonempty subset $A$ of $\mathbb{N}$ we will denote by $\big<A\big>$ the submonoid of $(\mathbb{N},+)$ generated by $A$, that is,
\begin{displaymath}
\big<A\big> = \{\lambda_1 a_1 + \cdots + \lambda_n a_n | n \in \mathbb{N}\texttt{\char`\\}\{0\}, a_i \in A, \lambda_i \in \mathbb{N}
\end{displaymath}
\textrm{ for all } $i \in \{1,\cdots,n\}\}$.
\end{definition}
The greatest integer that does not belong to a numerical semigroup $S$ is called the Frobenius number of $S$ and is denoted by $F(S)$. In other words, the Frobenius number is the largest integer that cannot be expressed as a sum $\sum_{i=1}^n t_i a_i$, where $t_1, t_2, ..., t_n$ are nonnegative integers and $a_1, a_2, ..., a_n$ are given positive integers such that $\gcd(a_1, a_2, ..., a_n) = 1$. Finding the Frobenius number is called the Frobenius problem, the coin problem or the money changing problem. The Frobenius problem is not only interesting for pure mathematicians but is also connected with graph theory in \cite{Heap (1965), Hujter (1987)} and the theory of computer science in \cite{Raczunas1996}, as introduced in \cite{Owens2003}.
There are explicit formulas for calculating the Frobenius number when only two relatively prime numbers are present \cite{Sylvester1883}. Recently, semi-explicit formula \cite{Robles2012} and a special case \cite{Tripathi (2013)} for the Frobenius number for three relatively prime numbers are presented. And finally, completed explicit formula was presented in 2017\cite{Tripathi (2017)}.
\\
\\
F. Curtis proved in \cite{Curtis (1990)} that the Frobenius number for three or more relatively prime numbers cannot be given by a finite set of polynomials and Ram\'irez-Alfons\'in proved in \cite{Ramirez1996} that the problem is NP-hard. Currently, only algorithmic methods for determining the general formula for the Frobenius number of a set that has three or more relatively prime numbers in \cite{Beihoffer (2005), Bocker (2007)} exist. Some recent studies have reported that the running time for the fastest algorithm is $O(a_1)$, with the residue table in memory in \cite{Brauer (1962)} and $O(na_1)$ with no additional memory requirements in \cite{Bocker (2007)}.In addition, research on the limiting distribution in \cite{Shchur2009} and lower bound in \cite{Aliev (2007),Fel (2015)} of the Frobenius number were presented. From an algebraic viewpoint, rather than finding the general formula for three or more relatively prime numbers, the formulae for special cases were found such as the Frobenius number of a set of integers in a geometric sequence in \cite{Ong2008}, a Pythagorean triples in \cite{Gil (2015)} and three consecutive squares or cubes in \cite{Lepilov (2015)}. Recently, various methods for solving the Frobenius problem for numerical semigroups have been suggested in \cite{Aliev (2009),Rosales2000,Rosales2009,Rosales2004}, etc. In particular, a method for computing the Ap\'ery set and obtaining the Frobenius number using the Ap\'ery set is an efficient tool for solving the Frobenius problem of numerical semigroups as reported in \cite{Marquez (2015),Rosales2009,Ramirez2009}. Furthermore, in recent articles presenting the Frobenius problems for Fibonacci numerical semigroups in \cite{Marin (2007)}, Mersenne numerical semigroups in \cite{Rosales2017}, Thabit numerical semigroups in \cite{Rosales2015} and repunit numerical semigroups in \cite{Rosales2016}, this method is used to obtain the Frobenius number.
\\
\\
The Frobenius problem in the numerical semigroups \\
$\big<\{3\cdot 2^{n+i} - 1| i \in \{0,1,\cdots \}\}\big>$ for $n \in \{0,1,\cdots \}$ was presented in \cite{Rosales2015}. In \cite{Rosales2015}, the authors recall that a Thabit number $3 \cdot 2^n - 1$ and Thabit numerical semigroups $T(n) = \big< \{3 \cdot 2^{n+i} - 1 | i \in \{0,1,\cdots\}\}\big>$ for a nonnegative integer $n$ and they used the definition of the minimal system of generators for $T(n)$ as the smallest subset of $\big< \{3 \cdot 2^{n+i} - 1 | i \in \mathbb{N}\}\big>$ that equals $T(n)$. In \cite{Rosales2015}, the minimal system of generators for $T(n)$ is $\big< \{3 \cdot 2^{n+i} - 1 | i \in \{0,1,\cdots ,n+1\}\}\big>$, and it is unique. The embedding dimension is the cardinality of the minimal system of generators. By the minimality of the system $\big< \{3 \cdot 2^{n+i} - 1 | i \in \{0,1,\cdots ,n+1\}\}\big>$ for $T(n)$, the embedding dimension for $T(n)$ is $n+2$. For any set $S$ and $x \in S\texttt{\char`\\}\{0\}$, the Ap$\acute{e}$ry set was defined by $Ap(S,x) = \{s \in S | s - x \not\in S\}$. Let $s_i = 3 \cdot 2^{n+i} - 1$ for each nonnegative integer $i$. Then, the Ap\'{e}ry set is defined by $Ap(T(n),s_0) = \{s \in T(n) | s-s_0 \not \in T(n) \}$ for $s_0$. In \cite{Rosales2015}, $Ap(T(n),s_0)$ was described explicitly leading to a solution to the Frobenius problem.  Let $R(n)$ be the set of sequences $(t_1,\cdots,t_{n+1}) \in \{0,1,2\}^{n+1}$ that satisfy the following conditions:
\begin{enumerate}
\item $t_{n+1} \in \{0,1\}$,
\item If $t_j = 2$, then $t_i = 0$ for all $i < j \leq n$,
\item If $t_n = 2$, then $t_{n+1} = 0$,
\item If $t_n = t_{n+1} = 1$, $t_i = 0$ for all $1\leq i <n$.
\end{enumerate}
Then \cite{Rosales2015} concludes that $Ap(T(n),s_0) = \{t_1 s_1 + \cdots + t_{n+1} s_{n+1} | (t_1,\cdots, t_{n+1}) \in R(n)\}$. The Frobenius number of the numerical semigroups was presented by $F(S) = \max(Ap(S,x)) - x$ in \cite{Rosales2009} and therefore the Frobenius number of Thabit numerical semigroups is $s_n + s_{n+1} - s_0 = 9\cdot 2^{2n} - 3\cdot 2^n - 1$. Also, an extended result of \cite{Rosales2015} have been suggested in 2017 which dealt the numerical semigroups \\
$\big<\{(2^k - 1)\cdot 2^{n+i} - 1| i \in \{0,1,\cdots \}\}\big>$ for $n \in \{0,1,\cdots \}$ and $2 \leq k \leq 2^n$ \cite{Gu (2017)}. In other words, the coefficient $3$ in Thabit numerical semigroups was extended. Also, we suggested an extended result of \cite{Rosales2015} which extended the coefficients $3$ and $1$ in Thabit numerical semigroups. The form of numerical semigroup is as follows: $\big<\{(2^k + 1)\cdot 2^{n+i} - (2^k - 1)| i \in \{0,1,\cdots \}\}\big>$ for $n \in \{0,1,\cdots \}$ and $k \in \{1,2,\cdots\}$.
\\
The first principal purpose of this paper is to solve the Frobenius problem for numerical semigroups generated by Thabit number base $b$ defined by $\{(b+1)\cdot b^{n+i} - 1 | i \in \{0,1,\cdots \} \}$  for $n \in \{0,1,\cdots \}$  to generalize the result in Thabit numerical semigroups in \cite{Rosales2015}. We define the minimal system of generators and the Ap\'{e}ry set of Thabit numerical semigroups base $b$. In addition, we discuss the Frobenius number and genus in these numerical semigroups. The paper in \cite{Rosales2015} is the case of $b = 2$ in this paper. \\
\\
The second purpose of this paper is to solve the Frobenius problem for numerical semigroups generated by Thabit number of the second kind base $b$ defined by $\{(b+1)\cdot b^{n+i} + 1 | i \in \{0,1,\cdots \} \}$  for $n \in \{0,1,\cdots \}$ and $b \not\equiv 1 \pmod{3}$.
\\
\\
The third purpose of this paper is to solve the Frobenius problem for numerical semigroups generated by Cunningham number defined by $\{b^{n+i} + 1 | i \in \{0,1,\cdots \} \}$  for $n \in \{0,1,\cdots \}$ and even positive integer $b$.
\\
\\
The fourth purpose of this paper is to solve the Frobenius problem for numerical semigroups generated by Fermat number base $b$ defined by $\{(b^{b^{n+i}} + 1 | i \in \{0,1,\cdots \} \}$  for $n \in \{0,1,\cdots \}$.
\\
\\
The remainder of this paper is organized as follows: 
In Section 2, we compute the minimal system of generators and the embedding dimension for Thabit numerical semigroups base $b$.
In Section 3, a method for obtaining the Ap\'{e}ry set and Frobenius number, genus for Thabit numerical semigroups base $b$.
In Section 4, we suggest the results of minimal system of generators, embedding dimension, the Apery set, and Frobenius number of Thabit numerical semigroups of the second kind base $b$ without the proofs.
In Section 5, we suggest the results of minimal system of generators, embedding dimension, the Apery set, and Frobenius number of the Cunningham numerical semigroups without the proofs.
In Section 6, we suggest the results of minimal system of generators, embedding dimension, the Apery set, and Frobenius number of the Fermat numerical semigroups base $b$ without the proofs.
Some theorems and definitions essential to understanding this paper are provided below.
\begin{theorem}\cite{Rosales2015,Rosales2009}
$\big<A\big>$ is a numerical semigroup if and only if $\gcd(A) = 1$.
\end{theorem}
\begin{definition}
A positive integer $x$ is a {\em Thabit number base $b$} if $x = (b+1)\cdot b^n - 1$ for some $n, b \in \mathbb{N}$ and $b \geq 2$.
\end{definition}
\begin{definition}
A numerical semigroup $S$ is called a {\em Thabit numerical semigroup base $b$} if there exists $n, b \in \mathbb{N}$ and $b \geq 2$ such that $S = \big< \{(b+1)\cdot b^{n+i} - 1 | i \in \mathbb{N}\}\big>$. We will denote by $T_b(n)$ the Thabit numerical semigroup base $b$ $\big< \{(b+1)\cdot b^{n+i} - 1 | i \in \mathbb{N}\}\big>$.
\end{definition}
\begin{definition}
A positive integer $x$ is a {\em Thabit number of the second kind base $b$} if $x = (b+1)\cdot b^n + 1$ for some $n, b \in \mathbb{N}$ and $b \geq 2$.
\end{definition}
\begin{definition}
A numerical semigroup $S$ is called a {\em Thabit numerical semigroup of the second kind base $b$} if there exists $n, b \in \mathbb{N}$, $b \geq 2$ and $b \not\equiv 1 \pmod{3}$ such that $S = \big< \{(b+1)\cdot b^{n+i} + 1 | i \in \mathbb{N}\}\big>$. We will denote by $T_{b'}(n)$ the Thabit numerical semigroup of the second kind base $b$ $\big< \{(b+1)\cdot b^{n+i} + 1 | i \in \mathbb{N}\}\big>$.
\end{definition}
\begin{definition}
We call a positive integer $x$ be a {\em Cunningham number} if $x = b^n + 1$ for some $n, b \in \mathbb{N}$, $2 | b$.
\end{definition}
\begin{definition}
A numerical semigroup $S$ is called a {\em Cunningham numerical semigroup} if there exists $n, b \in \mathbb{N}$ and $2 | b$ such that $S = \big< \{b^{n+i} + 1 | i \in \mathbb{N}\}\big>$. We will denote by $SC^{+}(b,n)$ the Cunningham numerical semigroup $\big< \{b^{n+i} + 1 | i \in \mathbb{N}\}\big>$.
\end{definition}
\begin{definition}
We call a positive integer $x$ be a {\em Fermat number base $b$} if $x = b^{b^n} + 1$ for some $n, b \in \mathbb{N}$ and $2 | b$.
\end{definition}
\begin{definition}
A numerical semigroup $S$ is called a {\em Fermat numerical semigroup base $b$} if there exists $n, b \in \mathbb{N}$ and $2 | b$ such that $S = \big< \{b^{b^{n+i}} + 1 | i \in \mathbb{N}\}\big>$. We will denote by $SF(b,n)$ the Fermat numerical semigroup $\big< \{b^{b^{n+i}} + 1 | i \in \mathbb{N}\}\big>$.
\end{definition}
\begin{definition}\cite{Rosales2015,Rosales2009}
If $S$ is a numerical semigroup and $S = \big< A\big>$ then we say that $A$ is a {\em system of generators of} $S$. Moreover, if $S \neq \big< X\big>$ for all $X \not \subseteq A$, we say that $A$ is a {\em minimal system of generators of} $S$.
\end{definition}
\begin{theorem}\cite{Rosales2009}
Every numerical semigroup admits a unique minimal system of generators, which in addition is finite.
\end{theorem}
\begin{definition}\cite{Rosales2015,Rosales2009}
We call the cardinality of its minimal system of generators the {\em embedding dimension of} $S$, denoted by $e(S)$.
\end{definition}
\begin{definition}\cite{Rosales2015,Rosales2009}
We call the greatest integer that does not belong to $S$ the {\em Frobenius number of} $S$ denoted by $F(S)$ and the cardinality of $\mathbb{N}\texttt{\char`\\}S$ is called the {\em genus of} $S$ and denoted by $g(S)$.
\end{definition}
\section{The embedding dimension for $T_b(n)$} \label{sec_The embedding dimension for $GT(n,k)$}
If $n,b \in \mathbb{N}$ and $b \geq 2$, then $T_b(n)$ is a submonoid of $(\mathbb{N},+)$. Moreover we have $\{(b+1)\cdot b^n -1, (b+1)\cdot b^{n+1} - 1\} \subseteq T_b(n)$ and $\gcd((b+1)\cdot b^n - 1, (b+1)\cdot b^{n+1} - 1) = \gcd((b+1)\cdot b^{n+1} - b, (b+1)\cdot b^{n+1} - 1)  | b-1$. But $(b+1)\cdot b^n - 1 \equiv 1 \pmod{b-1}$ implies that if we let $(b+1)\cdot b^n - 1 = g\alpha = h(b-1) + 1$ where $g = \gcd((b+1)\cdot b^n - 1,(b+1)\cdot b^{n+1} - 1) | b-1$ and $h \in \mathbb{N}$ then $g | 1$ and hence $\gcd(T_b(n)) = 1$ and $T_b(n)$ is a numerical semigroup.
\begin{lemma}\label{Lem_2a_2m_eq}
Let $A$ be a nonempty set of positive integers, $b \in \mathbb{N}, b \geq 2$ and $M = \big<A\big>$. Then the following conditions are equivalent:
\begin{enumerate}
\item $ba+(b-1) \in M$ for all $a \in A$,
\item $bm+(b - 1) \in M$ for all $m \in M\texttt{\char`\\}\{0\}$.
\end{enumerate}
\end{lemma}
The proof of the above lemma is similar to that of lemma 1 in \cite{Rosales2015}, and it is a special case of the lemma 2 in \cite{Rosales2016}.
\begin{proposition}
If $n,b \in \mathbb{N}$ and $b \geq 2$, then $bt+(b - 1) \in T_b(n)$ for all $t \in T_b(n)\texttt{\char`\\}\{0\}$.
\end{proposition}
The proof of the above proposition is similar to that of proposition 2 in \cite{Rosales2015}.\\
We need some preliminary results to find out the minimal system of generators of $T_b(n)$.
\begin{lemma}\label{Lem_2s}
Let $n,b \in \mathbb{N}$ and $b \geq 2$ and
$S = \big<\{(b+1)\cdot b^{n+i} - 1| i \in \{0,1,\cdots,n+1\}\}\big>$. Then $bt + (b-1) \in S$ for all $t \in S\texttt{\char`\\}\{0\}$.
\end{lemma}
\begin{proof}
The proof of the above lemma is similar to that of lemma 3 in \cite{Rosales2015} and the result can be obtained easily by the equation as follows: 
\begin{align*}
& b\{(b+1)\cdot b^{2n+1} - 1\} + (b-1) \\
= & (b+1)\cdot b^{2n+2} - 1 \\
= & ((b-1)(b+1)\cdot b^n - 1)((b+1)\cdot b^n - 1) + (b+1)\cdot b^{n+1} - 1 + (b+1)\cdot b^{2n} - 1.
\end{align*}
\end{proof}
A system of generators of $GT(n,k)$ will be given in the next lemma; note that it is not a minimal system of generators in the general case.
\begin{lemma}\label{Lem_rev_minimal_1}
If $n,b \in \mathbb{N}$ and $b \geq 2$, then $T_b(n) \\
=  \big<\{(b+1)\cdot b^{n+i} - 1| i \in \{0,1,\cdots,n+1\}\}\big>$.
\end{lemma}
\begin{proof}
This proof is similar to the proof of Lemma 4 in \cite{Rosales2015}.
\end{proof}
In addition, we suggest a conclusion for a minimal system of generators of $T_b(n)$ in the following theorem.
\begin{theorem}\label{Thm_minimal}
If $n,b \in \mathbb{N}$ and $b \geq 2$, then $\big<\{(b+1)\cdot b^{n+i} - 1| i \in \{0,1,\cdots,n+1\}\}\big>$ is a minimal system of generators.
\end{theorem}
\begin{proof}
Let us suppose conversely, that $(b+1)\cdot b^{2n+1} - 1 \in \big<\{(b+1)\cdot b^{n+i} - 1 | i \in \{0,1,\cdots, n\}\}\big>$. Then there exist $a_0, \cdots, a_{n+k-1} \in \mathbb{N}$ such that
\begin{align*}
(b+1)\cdot b^{2n+1} - 1 & = \sum_{j=0}^{n} a_j \{(b+1)\cdot b^{n+j} - 1\} \\ 
& = \sum_{j=0}^{n} (b + 1) a_j b^{n+j} - \sum_{j=0}^{n} a_j
\end{align*}
and consequently, $\sum_{j=0}^{n} a_j \equiv 1 \mod{(b+1)\cdot b^{n}}$. Hence $\sum_{j=0}^{n} a_j = 1 + t \cdot (b+1) \cdot b^n$ for some $t \in \mathbb{N}$. In addition, it is clear that $t \neq 0$ and thus $\sum_{j=0}^{n} a_j \geq 1 + (b+1)\cdot b^n$. Combining these results, we obtain the inequality
\begin{gather*}
\sum_{j=0}^{n} a_j \{(b+1)\cdot b^{n+j} - 1\} > \{(b + 1)\cdot b^{2n+1} - 1 \}.
\end{gather*}
This completes the proof.
\end{proof}
By Theorem \ref{Thm_minimal}, we can identify the embedding dimension of $T_b(n)$ for all  $n,b \in \mathbb{N}$ and $b \geq 2$.
\begin{corollary}
Let $n,b \in \mathbb{N}$ and $b \geq 2$ and let $T_b(n)$ be a Thabit numerical semigroup base $b$ associated with $n$ and $b$. Then $e(T_b(n)) = n + 2$.
\end{corollary}
Gathering all this information we obtain that for each integer $t$ greater than or equal to two, there exists a Thabit numerical semigroup base $b$ $T_b(n)$ with embedding dimension $t$.
\begin{example}
$T_b(0) = \big<\{(b+1) \cdot b^0 - 1, (b + 1) \cdot b^1 - 1 \}\big> = \big<\{b, b^2 + b - 1\}\big>$ is a Thabit numerical semigroup base $b$ with embedding dimension $0 + 2 = 2$.
\end{example}

\begin{example}
$T_b(3)= \big<\{(b+1) \cdot b^3 - 1, (b + 1) \cdot b^4 - 1 \}, (b+1) \cdot b^5 - 1, (b + 1) \cdot b^6 - 1, (b+1) \cdot b^7 - 1 \} \big> = \big<\{b^4 + b^3 - 1, b^5 + b^4 - 1, b^6 + b^5 - 1, b^7 + b^6 - 1, b^8 + b^7 - 1\}\big>$ is a Thabit numerical semigroup base $b$ with embedding dimension $3 + 2 = 5$.
\end{example}
\section{The Ap\'{e}ry set for $T_b(n)$}\label{sec_The Apery set for $GT(n,k)$}
Let $S$ be a numerical semigroup and let $x \in S \texttt{\char`\\}\{0\}$. The Ap\'{e}ry set of $x$ in $S$ is defined as $Ap(S,x) = \{s \in S | s-x \not \in S \}$ in \cite{Rosales2009}. From \cite{Rosales2009}, we easily deduce the following Lemma.
\begin{lemma} \cite{Rosales2009}
Let $S$ be a numerical semigroup and let $x \in S \texttt{\char`\\}\{0\}$. Then $Ap(S,x)$ has cardinality equal to $x$. Moreover $Ap(S,x) = \{w(0),w(1),\cdots, w(x-1)\}$ where $w(i)$ is the least element of $S$ congruent with $i$ modulo $x$ for all $i \in \{0,\cdots, x-1\}$.
\end{lemma}
\begin{example}
Let $S = \left<\{7,11,13\}\right>$. Then\\
$S = \{ 0,7,11,13,14,18,20,21,22, 24, 25, 26, 27,28,29,31,\rightarrow\}$
where the symbol $\rightarrow$ means that every integer greater than $31$ belongs to the set. \\
Hence $Ap(S,7) = \{0,11,13,22,24,26,37\}$.
\end{example}
The relation among the Frobenius number, genus and Ap\'{e}ry set of a numerical semigroup is provided in the following lemma.
\begin{lemma} \cite{Rosales2009, Selmer1977} \label{lem_F_g}
Let $S$ be a numerical semigroup and let $x \in S\texttt{\char`\\}\{0\}$. Then,
\begin{enumerate}
\item $F(S) = \max(Ap(S,x)) - x$
\item $g(S) = \frac{1}{x} (\sum_{w \in Ap(S,x)} w) - \frac{x-1}{2}$
\end{enumerate}
\end{lemma}
Henceforth, we will denote by $s_i$ the elements $(b+1)\cdot b^{n+i} - 1$ for each $i \in \{0,1,\cdots,n+1 \}$. Thus, with this notation, $\{s_0,s_1,\cdots,s_{n+1}\}$ is the minimal system of generators of $T_b(n)$.
\begin{lemma}\label{Lem_push_coefficients_1}
Let $n,b \in \mathbb{N}$ and $b \geq 2$. Then:
\begin{enumerate}
\item If $0<i\leq j<n+1 $ then $s_i + bs_j = bs_{i-1} + s_{j+1}$.
\item If $0<i\leq n+1 $ then
\end{enumerate}
\begin{displaymath}
s_i + bs_{n+1} = bs_{i-1} + (b-1)s_0^2 + (b-2)s_0 + s_1 + s_n.
\end{displaymath}
\end{lemma}
\begin{proof} (1) The proof is similar to that of (1) of Lemma 9 in \cite{Rosales2015}.\\
(2) It can be derived directly by the equation as follows:
\begin{align*}
& (b+1)\cdot b^{2n+2} - 1 \\
= & ((b-1)(b+1)\cdot b^n - 1)((b+1)\cdot b^n - 1) + (b+1)\cdot b^{n+1} - 1 + (b+1)\cdot b^{2n} - 1.
\end{align*}
\end{proof}
By lemma \ref{Lem_push_coefficients_1}, we can consider the set of coefficients $(t_1,\cdots, t_{n+1})$ such that the expressions $\sum_{j=1}^{n+1} t_j s_j$ represent all elements in $Ap(T_b(n),s_0)$. We will follow a step by step approach to establish the set of coefficients $(t_1,\cdots,t_{n+1})$. First, we obtain the set of coefficients $(t_1,\cdots , t_{n+1})$ such that $\sum_{j=1}^{n+1} t_j s_j$ contains all elements that are in $Ap(T_b(n),s_0)$, but that might not be equal.  We can obtain the set by the following lemma.
\begin{lemma} \label{Lem_A}
Let $A_b(n)$ be the set of $(t_1, \cdots, t_{n+1}) \in \{0,1,\cdots,b\}^{n+1}$ such that $t_{n+1} \in \{0,1,\cdots,b-1\}$ and if $t_j = b$, then $t_i = 0$ for all $i < j$. Then $Ap(T_b(n),s_0) \subseteq \{\sum_{j=1}^{n+1} t_j s_j | (t_1,\cdots,t_{n+1}) \in A_b(n)\}$.
\end{lemma}
\begin{proof}
The overall proof is the same as that of the lemma 10 in \cite{Rosales2015} and I will show that $bs_{n+1} \not \in Ap(T_b(n),s_0)$ in corollary \ref{cor_n1_k2}.
\end{proof}
\begin{lemma}\label{Lem_x-1}
Let $n,b \in \mathbb{N}$ and $b \geq 2$. If $x \in T_b(n)$ and $x \not\equiv0 \mod{s_0}$, then $x - (b-1) \in T_b(n)$.
\end{lemma}
\begin{proof}
If $x \in T_b(n)$, then there exist $a_0,\cdots,a_{n+1} \in \mathbb{N}$ such that $x = \sum_{j=0}^{n+1} a_j s_j$. On the other hand, if $x \not\equiv 0 \mod{s_0}$ then there exists $i \in \{1,\cdots, n+1\}$ such that $a_i \neq 0$. Note that $s_i = (b+1)\cdot b^{n+i} - 1$ as defined in Section \ref{sec_The embedding dimension for $GT(n,k)$}. Hence
\begin{align*}
x - (b-1) & = \sum_{j=0, j\neq i}^{n+1} a_j s_j + (a_i - 1)s_i + (b+1)\cdot b^{n+i} - b\\
& = \sum_{j=0, j\neq i}^{n+1} a_j s_j + (a_i - 1)s_i + b\{(b + 1)\cdot b^{n+i-1} - 1\}\\
& = \sum_{j=0, j\neq i-1, i}^{n+1} a_j s_j + (a_{i-1} + b) s_{i-1} + (a_i - 1)s_i \in T_b(n).
\end{align*}
\end{proof}
Note that contrary to this lemma, $x \in T(n)$ implies that $x - 1 \in T(n)$ in \cite{Rosales2015}.
Lemma \ref{Lem_x-1} is very important since $\gcd(b-1, s_0) = 1$ by $s_0 \equiv 1(\mod{b-1})$, for any $x > (b-1)\{(b + 1)\cdot b^n - 1\} = (b - 1)s_0$, the set $\{x - i(b - 1)| i \in \{0,1,\cdots,(b + 1)\cdot b^n - 2\}\}$ is a complete system of residues modulo $s_0$ and hence we obtain the following lemma.
\begin{lemma}\label{Lem_maxAp_1}
If $n,b \in \mathbb{N}$ and $b \geq 2$ then
\begin{gather*}
w(s_0 - (b - 1)) = \max(Ap(T_b(n)),s_0)).
\end{gather*}
\end{lemma}
From this lemma and the fact that $s_i \equiv b^i - 1 \mod{s_0}$, we can prove the following corollary previously announced in lemma \ref{Lem_A}.
\begin{corollary}\label{cor_n1_k2}
If $n,b \in \mathbb{N}$ and $b \geq 2$ then $bs_{n+1} \not \in Ap(T_b(n),s_0)$.
\end{corollary}
\begin{proof}
We already know from the proof of lemma \ref{Lem_push_coefficients_1} that \\
$s_i + bs_{n+1} \not \in Ap(T_b(n),s_0)$ for all $i \in \mathbb{N}$. Therefore, if $bs_{n+1} \in Ap(T_b(n),s_0)$, it can be expected to be a maximal element. Since $s_i \equiv b^i - 1(\mod{s_0})$, $bs_{n+1} \equiv b(b^{n+1} - 1) \equiv b^n - 1(\mod{s_0})$ but since $s_0 - (b-1) = (b+1)\cdot b^n - b \neq b^n - 1$, $bs_{n+1} \not\in Ap(T_b(n),s_0)$.
\end{proof}
By combining lemma \ref{Lem_maxAp_1} and corollary \ref{cor_n1_k2}, we obtain the following corollary.
\begin{corollary}\label{cor_maxAp_2}
\begin{gather*}
\max(Ap(T_b(n)),s_0)) = (b-1)s_n + (b-1)s_{n+1}.
\end{gather*}
\end{corollary}
\begin{proof}
By corollary \ref{cor_n1_k2}, we observe that the condition $t_{n+1} \leq b-1$ should be satisfied to $\sum_{j=1}^{n+1} t_j s_j \in Ap(T_b(n))$ where $(t_1,\cdots,t_{n+1}) \in A_b(n)$. Since $(b-1)s_n + (b-1)s_{n+1} \equiv -(b-1) \pmod{s_0}$ and it is the maximal number of the form $\sum_{j=1}^{n+1} t_j s_j$, it completes the proof.
\end{proof}
By combining corollary \ref{cor_maxAp_2} and lemma \ref{lem_F_g}, we obtain the Frobenius number of numerical semigroups generated by Thabit number base $b$.
\begin{theorem}
\begin{gather*}
F(T_b(n)) = (b^{3} + b^{2} - b - 1)\cdot b^{2n} - (b+1)\cdot b^{n} - 2b + 3.
\end{gather*}
\end{theorem}
\begin{example}
Let $b = 2$. Then we obtain
\begin{gather*}
F(T_2(n)) = 9\cdot 2^{2n} - 3\cdot 2^{n} - 1.
\end{gather*}
It is the Frobenius number of Thabit numerical semigroups suggested in \cite{Rosales2015}. 
\end{example}
Let $R_b(n)$ be the set of the sequences $(t_1,\cdots,t_{n+1}) \in A_b(n)$ that if $t_{n+1} = b-1$, it satisfies the following conditions:
\begin{enumerate}
\item $t_n \leq b-1$.
\item If $t_{n} = b-1, t_1 = \cdots = t_{n-1} = 0$.
\end{enumerate}
Then, we obtain the following lemma:
\begin{lemma}\label{Lem_Ap_k<n}
\begin{gather*}
Ap(T_b(n), s_0) = \{\sum_{j=1}^{n+1} t_{j} s_{j} | (t_1,\cdots,t_{n+1}) \in R (n)\}.
\end{gather*}
\end{lemma}
\begin{proof}
Note that $Ap(T_b(n), s_0) \subseteq \{\sum_{j=1}^{n+1} t_{j} s_{j} | (t_1,\cdots,t_{n+1}) \in R_b(n)\}$ and hence \\
$\#\{\sum_{j=1}^{n+1} t_{j} s_{j} | (t_1,\cdots,t_{n+1}) \in R_b(n)\} \leq \# R_b(n)$. Then, $\# R_b(n)$ having the cardinality $s_0 = (b + 1) \cdot b^n - 1$ suffices  for the proof. We classify the cases to obtain the cardinality, as follows:
\begin{enumerate}
\item If $t_{n+1} \leq b-2$, then it can be again classified into two cases, as follows: \label{2<k<n_case_1}
\begin{enumerate}
\item Let $b \not \in \{t_1, \cdots, t_{n}\}$. Then
$\#\{(t_1,\cdots,t_{n+1}) \in R_b(n)\} = b^{n}(b-1)$ since $t_i \in \{0,1,\cdots,b-1\}$ for all $1 \leq i \leq n$ for each $t_{n+1} \in \{0,1,\cdots,b-2\}$.
\item Let $b \in \{t_1, \cdots, t_{n}\}$ \label{2<k<n_case_1_2}. If $t_i = b$ for some $i \in \{1,\cdots,n\}$, then $t_j = 0$ for all $j < i$ and $t_j \in \{0,1,\cdots,b-1\}$ for all $i < j \leq n$. Hence, $\#R_b(n)$ in this case is $b^{n - i}$. Thus, we use summation to determine $\#R_b(n)$ for (\ref{2<k<n_case_1_2}): $(b-1)\sum_{i=1}^{n} b^{n - i} = b^n - 1$.
\end{enumerate}
Hence, $\#R_b(n) = b^{n+1} - 1$ in case \ref{2<k<n_case_1}.
\item If $t_{n+1} = b-1$, then it can be again classified into two cases as follows:
\begin{enumerate}
\item Let $t_{n} = b-1$. Then the only possible case is $(t_1,\cdots,t_{n-1})=(0,\cdots,0)$.
\item Let $t_{n} \leq b-2$. Then $\#\{(t_1,\cdots,t_{n+1}) \in R_b(n)\} = b^{n} - 1$ by substituting $n-1$ instead $n$ in the case \ref{2<k<n_case_1}.
\end{enumerate}
\end{enumerate}
Hence, we can conclude that $\#R_b(n) = (b^{n+1} - 1) + 1 + (b^n -1) = s_0$.
\end{proof}
To obtain the genus of the numerical semigroups generated by Thabit number base $b$, we have to check the number of elements in $R_b(n)$ when one element $t_i$ is fixed.
\begin{lemma}\label{lem_Lem27}
Let $i \in \{1,2,\cdots,n+1\}$ where $n \geq 2$ be an integer. Then,
\begin{gather*}
\#\{(t_1,\cdots,t_{n+1}) \in R_b(n) | t_i = b\} = 
\left\{ \begin{array}{ll}
(b^2 - 1)\cdot b^{n-i-1} & \text{if} ~~ i \in \{1,\cdots,n-1\},\\
b-1 & \text{if} ~~ i = n,\\
0 & \text{if} ~~ i = n+1.
\end{array} \right.
\end{gather*}
\end{lemma}
\begin{proof}
If $i \in \{1,\cdots,n-1\}, (t_1,\cdots,t_{n+1}) \in R_b(n)$ and $t_i = b$, then $t_1 = \cdots = t_{i-1} = 0, t_{i+1}, \cdots, t_{n+1} \in \{0,1,\cdots,b-1\}$ and furthermore either $t_n \leq b-2$ or $t_{n+1} \leq b-2$. Hence $\#\{(t_1,\cdots,t_{n+1}) \in R_b(n) | t_i = b\} = (b^2 - 1) b^{n-i-1}$. And $\{(t_1,\cdots,t_{n+1}) \in R_b(n) | t_n = b\} = \{(0,\cdots,0,b,0), \cdots, (0,\cdots,0,b,b-2)\}$ and $\{(t_1,\cdots,t_{n+1}) \in R_b(n) | t_{n+1} = b\} = \emptyset$.
\end{proof}
\begin{lemma}\label{lem_Lem28}
Let $i \in \{1,2,\cdots,n-1\}$ where $n \geq 2$ be an integer. Then,
\begin{gather*}
\#\{(t_1,\cdots,t_{n+1}) \in R_b(n) | t_i = k\} = (b + 1)(b^{n-1} - b^{n-i-1})
\end{gather*}
for each $k \in \{1,\cdots,b-1\}$.
\end{lemma}
\begin{proof}
Let $i \in \{1,\cdots,n-1\}$. Then it can be classified into two cases as follows
\begin{enumerate}
\item If $b \not\in \{t_1,\cdots, t_{i-1}\}$, then $t_1,\cdots, t_{i-1},t_{i+1},\cdots,t_{n+1} \in \{0,1,\cdots,b-1\}$ and either $t_n \leq b-2$ or $t_{n+1} \leq b-2$. Therefore $\#\{(t_1,\cdots,t_{n+1}) \in R_b(n) | t_i = k \textrm{ and } b \not\in\{t_1,\cdots,t_{i-1}\}\} = (b^2 - 1) b^{n-2}$ for each $k$.
\item If $b \in \{t_1,\cdots,t_{i-1}\}$, then $t_j = b$ for some $j \in \{1,\cdots,i-1\}$. Thus $t_1 = \cdots = t_{j-1} = 0, t_{j+1}, \cdots, t_{n+1} \in \{0,1,\cdots,b-1\}, t_i = k$ and either $t_n \leq b-2$ or $t_{n+1} \leq b-2$. Hence $\#\{(t_1,\cdots,t_{n+1}) \in R_b(n) | t_i = k \textrm{ and } t_j = b\} = (b^2 - 1)b^{n-j-2}$ for each $k$.
\end{enumerate}
Hence, $\#\{(t_1,\cdots,t_{n+1}) \in R_b(n) | t_i = k\} = (b^2 - 1) b^{n-2} + \sum_{j=1}^{i-1} (b^2 - 1)b^{n-j-2} = (b^2 - 1)b^{n-2} + (b + 1)(b^{n-2} - b^{n-i-1}) = (b+1)(b^{n-1} - b^{n-i-1}) $ for each $k \in \{1,\cdots,b-1\}$.
\end{proof}
\begin{lemma}\label{lem_Lem28_2}
Let $n \geq 2$ be an integer. Then,
\begin{gather*}
\#\{(t_1,\cdots,t_{n+1}) \in R_b(n) | t_n = k\} = \frac{b^{n+1} - b}{b-1}
\end{gather*}
for each $k \in \{1,\cdots,b-2\}$.
\end{lemma}
\begin{proof}
Let $i = n$. Then it can be classified into two cases as follows
\begin{enumerate}
\item If $b \not\in \{t_1,\cdots,t_{n-1}\}$, then $t_1,\cdots,t_{n-1},t_{n+1} \in \{0,1,\cdots,b-1\}$. Hence, $\#\{(t_1,\cdots,t_{n+1}) \in R_b(n) | t_n = k \textrm{ and } b \not\in \{t_1,\cdots,t_{n-1}\}\} = b^n$ for each $k$.
\item If $b \in \{t_1,\cdots,t_{n-1}\}$, then $t_j = b$ for some $j \in \{1,\cdots,n-1\}$. Then $t_1 = \cdots = t_{j-1} = 0$, $t_{j+1}, \cdots, t_{n+1} \in \{0,1,\cdots,b-1\}$. Whence $\#\{(t_1,\cdots,t_{n+1}) \in R_b(n) | t_n = k \textrm{ and } t_j = b\} = b^{n-j}$ for each $k$.
\end{enumerate}
Hence, $\#\{(t_1,\cdots,t_{n+1}) \in R_b(n) | t_n = k\} = b^{n} + \sum_{j=1}^{n-1} b^{n-j} = b^{n} + b^{n-1} + \cdots + b = \frac{b^{n+1} - b}{b-1}$ for each $k \in \{1,\cdots,b-2\}$.
\end{proof}

\begin{lemma}\label{lem_Lem28_3}
Let $n \geq 2$ be an integer. Then,
\begin{gather*}
\#\{(t_1,\cdots,t_{n+1}) \in R_b(n) | t_n = b - 1\} = b^n.
\end{gather*}
\end{lemma}
\begin{proof}
Let $i = n$. Then it can be classified into two cases as follows
\begin{enumerate}
\item If $b \not\in \{t_1,\cdots,t_{n-1}\}$, then $t_1,\cdots,t_{n-1},t_{n+1} \in \{0,1,\cdots,b-1\}$. Besides if $t_{n+1} = b-1$, then $t_1 = \cdots = t_{n-1} = 0$. Hence, $\#\{(t_1,\cdots,t_{n+1}) \in R_b(n) | t_n = b-1 \textrm{ and } b \not\in \{t_1,\cdots,t_{n-1}\}\} = (b-1)b^{n-1} + 1$.
\item If $b \in \{t_1,\cdots,t_{n-1}\}$, then $t_j = b$ for some $j \in \{1,\cdots,n-1\}$. Then $t_1 = \cdots = t_{j-1} = 0$, $t_{j+1}, \cdots, t_{n+1} \in \{0,1,\cdots,b-1\}$ and $t_{n+1} \leq b-2$. Whence $\#\{(t_1,\cdots,t_{n+1}) \in R_b(n) | t_n = b-1 \textrm{ and } t_j = b\} = (b-1)b^{n-j-1}$.
\end{enumerate}
Hence, $\#\{(t_1,\cdots,t_{n+1}) \in R_b(n) | t_n = b-1\} = (b-1)b^{n-1} + 1 + (b-1)\sum_{j=1}^{n-1} b^{n-j-1} = (b-1)b^{n-1} + 1 + b^{n-1} - 1 = b^n$.
\end{proof}

\begin{lemma}\label{lem_Lem28_4}
Let $n \geq 2$ be an integer. Then,
\begin{gather*}
\#\{(t_1,\cdots,t_{n+1}) \in R_b(n) | t_{n+1} = k\} = \frac{b^{n+1} - 1}{b-1}
\end{gather*}
for each $k \in \{1,\cdots,b-2\}$.
\end{lemma}
\begin{proof}
Let $i = n+1$. Then it can be classified into two cases as follows
\begin{enumerate}
\item If $b \not\in \{t_1,\cdots,t_{n}\}$, then $t_1,\cdots,t_{n} \in \{0,1,\cdots,b-1\}$. Hence, $\#\{(t_1,\cdots,t_{n+1}) \in R_b(n) | t_{n+1} = k \textrm{ and } b \not\in \{t_1,\cdots,t_{n}\}\} = b^n$ for each $k$.
\item If $b \in \{t_1,\cdots,t_{n}\}$, it can be again classified into two cases as follows
\begin{enumerate}
\item If $t_j = b$ for some $j \in \{1,\cdots,n-1\}$, $t_1 = \cdots = t_{j-1} = 0$, $t_{j+1},\cdots, t_{n} \in \{0,1,\cdots,b-1\}$. Whence $\#\{(t_1,\cdots,t_{n+1}) \in R_b(n) | t_{n+1} = k \textrm{ and } t_j = b\} = b^{n-j}$ for each $k$.
\item If $t_{n} = b$, $t_1 = \cdots = t_{n-1} = 0$. Hence $\#\{(t_1,\cdots,t_{n+1}) \in R_b(n) | t_{n+1} = k \} = 1$ for each $k$.
\end{enumerate}  
\end{enumerate}
Hence, $\#\{(t_1,\cdots,t_{n+1}) \in R_b(n) | t_{n+1} = k\} = b^{n} + \sum_{j=1}^{n-1} b^{n-j} + 1 = b^n + b^{n-1} + \cdots + 1 = \frac{b^{n+1} - 1}{b-1}$ for each $k \in \{1,\cdots,b-2\}$.
\end{proof}

\begin{lemma}\label{lem_Lem28_5}
Let $n \geq 2$ be an integer. Then,
\begin{gather*}
\#\{(t_1,\cdots,t_{n+1}) \in R_b(n) | t_{n+1} = b - 1\} = b^{n}.
\end{gather*}
\end{lemma}
\begin{proof}
Let $i = n+1$. Then it can be classified into two cases as follows
\begin{enumerate}
\item If $b \not\in \{t_1,\cdots,t_{n}\}$, then $t_1,\cdots,t_{n} \in \{0,1,\cdots,b-1\}$. Besides if $t_{n} = b-1$, then $t_1 = \cdots = t_{n-1} = 0$. Hence, $\#\{(t_1,\cdots,t_{n+1}) \in R_b(n) | t_n = b-1 \textrm{ and } b \not\in \{t_1,\cdots,t_{n-1}\}\} = (b-1)b^{n-1} + 1$.
\item If $b \in \{t_1,\cdots,t_{n}\}$, then $t_j = b$ for some $j \in \{1,\cdots,n-1\}$. Then $t_1 = \cdots = t_{j-1} = 0$, $t_{j+1}, \cdots, t_{n-1} \in \{0,1,\cdots,b-1\}$, $t_n \leq b-2$ and $t_{n+1} = b - 1$. Hence $\#\{(t_1,\cdots,t_{n+1}) \in R_b(n) | t_{n+1} = b-1 \textrm{ and } t_j = b\} = (b-1)b^{n-j-1}$.
\end{enumerate}
Hence, $\#\{(t_1,\cdots,t_{n+1}) \in R_b(n) | t_{n+1} = b-1\} = (b-1)b^{n-1} + 1 + (b-1)\sum_{j=1}^{n-1} b^{n-j-1} = (b-1)b^{n-1} + 1 + b^{n-1} - 1 = b^{n}$.
\end{proof}
By combining the above lemmas, we obtain the genus of $T_b(n)$.
\begin{theorem}
Let $n,b \in \mathbb{N}$ and $b \geq 2$ and $T_b(n)$ be the Thabit numerical semigroup base $b$ associated to $n$. Then,
\begin{gather*}
g(T_b(n)) = \frac{(b^3 + b^2 - b - 1)b^{2n} + \{(n-1)(b^2 - 1) - 2\}b^n - 2b + 4}{2}
\end{gather*}
\end{theorem}
\begin{proof}
Suppose that $n \geq 2$ and from lemma \ref{lem_F_g}, we know the formula for genus of $T_b(n)$ as follows
\begin{gather*}
g(T_b(n)) = \frac{1}{s_0} (\sum_{(t_1,\cdots,t_{n+1} \in R_b(n)} t_1 s_1 + \cdots + t_{n+1}s_{n+1}) - \frac{s_0-1}{2}.
\end{gather*}
We start with
\begin{align*}
& \sum_{(t_1,\cdots,t_{n+1}) \in R_b(n)} (t_1 s_1 + \cdots + t_{n+1} s_{n+1}) \\
= & \sum_{k=1}^{b-1} \sum_{(t_1,\cdots,t_{n+1}) \in R_b(n), t_1 = k} ks_1 +  \cdots +  \sum_{k=1}^{b-1} \sum_{(t_1,\cdots,t_{n+1}) \in R_b(n), t_{n-1} = k} ks_{n-1} \\
& + \quad \sum_{k=1}^{b-2} \sum_{(t_1,\cdots,t_{n+1}) \in R_b(n), t_{n} = k} ks_n +  \sum_{(t_1,\cdots,t_{n+1}) \in R_b(n), t_{n} = b-1} (b-1)s_n \\
& + \sum_{k=1}^{b-2} \sum_{(t_1,\cdots,t_{n+1}) \in R_b(n), t_{n+1} = k} ks_{n+1} +  \sum_{(t_1,\cdots,t_{n+1}) \in R_b(n), t_{n+1} = b-1} (b-1)s_{n+1} \\
& + \quad  \sum_{(t_1,\cdots,t_{n+1}) \in R_b(n), t_1 = b} bs_1 + \cdots + \sum_{(t_1,\cdots,t_{n+1}) \in R_b(n), t_{n} = b} bs_n.
\end{align*}
By lemmas \ref{lem_Lem28}, \ref{lem_Lem28_2}, \ref{lem_Lem28_3}, \ref{lem_Lem28_4}, \ref{lem_Lem28_5}, we obtain that 
\begin{align*}
& \sum_{(t_1,\cdots,t_{n+1}) \in R_b(n)} (t_1 s_1 + \cdots + t_{n+1} s_{n+1}) \\
= & \sum_{i=1}^{n-1} \frac{(b-1)b}{2} \cdot (b + 1)(b^{n-1} - b^{n-i-1}) \cdot \{(b+1)\cdot b^{n+i} - 1\} \\
& \quad + \frac{(b-2)(b-1)}{2} \cdot [ \frac{b^{n+1} - b}{b-1} \cdot \{(b+1)\cdot b^{2n} - 1\} + \frac{b^{n+1} - 1}{b-1} \cdot \{(b+1)\cdot b^{2n+1} - 1\} ]\\
& \quad + b^n (b-1) \cdot \{(b+1)\cdot (b^{2n} + b^{2n+1}) - 2\} \\
& \quad + \sum_{i=1}^{n-1} (b^2 - 1)b^{n-i} \{(b+1)\cdot b^{n+i} - 1\} + (b-1)b\{(b+1)\cdot b^{2n} - 1\} \\
= & \frac{(b-1)b(b+1)}{2} \{ (b+1)\cdot \frac{b^{3n-1} - b^{2n}}{b-1} - (n-1)b^{n-1} - (n-1)(b+1)b^{2n-1} + \frac{b^{n-1} - 1}{b-1}\} \\
& \quad + \frac{b-2}{2} \cdot [(b^{n+1} - b) \cdot \{(b+1)\cdot b^{2n} - 1\} + (b^{n+1} - 1) \cdot \{(b+1)\cdot b^{2n+1} - 1\}] \\
& \quad + b^n (b-1) \cdot \{(b+1)\cdot (b^{2n} + b^{2n+1}) - 2\} \\
& \quad + (b^2 - 1)(b+1)(n-1)b^{2n} - (b+1)(b^n - b) + (b-1)b\{(b+1)\cdot b^{2n} - 1\} \\
= & \frac{b^2 + 2b + 1}{2} (b^{3n} - b^{2n+1}) - \frac{(b^2 - 1)(n-1)}{2}b^{n} - \frac{(b^3 + b^2 - b - 1)(n-1)}{2}b^{2n} \\
& \quad + \frac{(b+1)(b^{n} - b)}{2} + \frac{(b-2)(b+1)}{2} b^{3n+1} - \frac{(b-2)(b+1)}{2}b^{2n+1} - \frac{b-2}{2}b^{n+1} \\
& \quad + \frac{(b-2)b}{2} + \frac{(b-2)(b+1)}{2} b^{3n+2} - \frac{(b-2)(b+1)}{2}b^{2n+1} - \frac{b-2}{2} b^{n+1} + \frac{b-2}{2} \\
& \quad + b^n (b-1)(b+1) b^{2n} + b^n (b-1)(b+1) b^{2n+1} - 2b^n (b-1) \\
& \quad + (b^3 + b^2 - b - 1)(n-1)b^{2n} - (b+1)b^n + 2b + (b^2-1)b^{2n+1} \\
= & \frac{b^4 + 2b^3 - 2b - 1}{2} b^{3n} + \frac{-b^3 + b + (n-1)(b^3 + b^2 - b - 1)}{2}b^{2n} \\
& \quad + \frac{-2b^2 - b + 3 - (b^2 - 1)(n-1)}{2}b^n + b-1 \\
= & \{(b+1)\cdot b^n - 1\} \{ \frac{b^3 + b^2 - b - 1}{2} b^{2n} + \frac{b-1 + (n-1)(b^2 - 1)}{2}b^n  -b + 1\} \\
= & \frac{1}{2} s_0 [(b^3 + b^2 - b - 1)b^{2n} + \{b-1 + (n-1)(b^2 - 1)\}b^n - 2b + 2].
\end{align*}
Hence,
\begin{align*}
g(T_b(n)) & = \frac{1}{2} [(b^3 + b^2 - b - 1)b^{2n} + \{b-1 + (n-1)(b^2 - 1)\}b^n - 2b + 2] - \frac{(b+1)\cdot b^n - 2}{2} \\
& = \frac{(b^3 + b^2 - b - 1)b^{2n} + \{(n-1)(b^2 - 1) - 2\}b^n - 2b + 4}{2}.
\end{align*}
\end{proof}
We summarize all of our results by suggesting an example.
\begin{example}
In the case of $n=1$, $T_b(1) = \big< s_0,s_1,s_2 \big> = \big< b^2 + b - 1, b^3 + b^2 - 1, b^4 + b^3 - 1\big>$ and we obtain
\begin{align*}
& Ap(T_b(1), s_0) \\
= & \{0,s_1, \cdots, bs_1, \\
& \quad s_2, s_1 + s_2 , \cdots, bs_1 + s_2, \\
& \quad 2s_2, s_1 + 2s_2, \cdots, bs_1 + 2s_2, \\
& \quad \cdots \\
& \quad (b-1)s_2, s_1 + (b-1)s_2 , \cdots, (b-1)s_1 + (b-1)s_2\}.
\end{align*}
Note that $\# Ap(T_b (1), s_0) = (b+1) + (b-1)(b+1) - 1 = b^2 + b - 1 = s_0$, $\max(Ap(T_b(1)),s_0) = (b-1)s_1 + (b-1)s_2$ and $F(T_b(1)) = (b-1)s_1 + (b-1)s_2 - s_0 = (b-1)(b^4 + 2b^3 + b^2 - 2) - (b^2 + b - 1) = (b^3 + b^2 - b - 1)b^2 - (b+1)b^1 - 2b + 3$ and $g(T_b(1)) = \frac{b^5 + b^4 - b^3 - b^2 - 2b + 4}{2}$. Let $b = 3$. Then we obtain the more detailed example as follows:
\begin{enumerate}
\item $T_3(1) = \big< 11, 35, 107\big>$. Note that $e(T_3(1)) = 3 = 1 + 2$.
\item $Ap(T_3(1), s_0) = \{0,s_1,2s_1,3s_1,s_2,s_1 + s_2, 2s_1 + s_2, 3s_1 + s_2,2s_2,s_1+2s_2,2s_1+2s_2\}$. Note that $\# Ap(T_3 (1), s_0) = 11 = s_0$.
\item $\max(Ap(T_3(1)),s_0) = 2 \cdot 35 + 2 \cdot 107 = 284$ and hence $F(T_3(1)) = \max(Ap(T_3(1)),s_0) - s_0 = 284 - 11 = 273$.
\item $g(T_3(1)) = \frac{3^5 + 3^4 - 3^3 - 3^2 - 2\cdot 3 + 4}{2} = 143$.
\end{enumerate}
\end{example}
\section{The Results of the semigroups $\big< \{(b+1)\cdot b^{n+i} + 1 | i \in \mathbb{N}\big>$ for $b \not\equiv 1\pmod{3}$}
\begin{theorem}\label{Thm_minimal_2}
If $n,b \in \mathbb{N}$, $b \geq 2$ and $b \not\equiv 1 \pmod{3}$, then $\big<\{(b+1)\cdot b^{n+i} + 1| i \in \{0,1,\cdots,n+1\}\}\big>$ is a minimal system of generators.
\end{theorem}
\begin{corollary}
Let $n,b \in \mathbb{N}$, $b \geq 2$ and $b \not\equiv 1 \pmod{3}$ and let $T_{b'}(n)$ be a Thabit numerical semigroup of the second kind base $b$ associated with $n$ and $b$. Then $e(T_b(n)) = n + 2$.
\end{corollary}
\begin{definition}
Let $b \geq 2$ and $b \not\equiv 1\pmod{3}$. Then $R_{b'}(n) = \{(t_1,t_2,\cdots,t_{n+1}) | t_i \in \{0,1,\cdots,b\}\}$ is defined by
\begin{enumerate}
\item If $t_i = b$, $t_j = 0$ for all $1 \leq j < i$.
\item $t_{n+1} \leq b-1$.
\item If $t_{n+1} = b-1$, then $t_n \leq b-1$ and if $(t_n, t_{n+1}) = (b-1,b-1)$, $t_1 \leq 2$ and all $t_i = 0$ for $i \neq 1,n,n+1$.
\end{enumerate}
\end{definition}
Then we obtain the following theorem:
\begin{theorem} \label{thm_Ap_explicit}
Let  $b \geq 2$ and $b \not\equiv 1\pmod{3}$. Then we obtain
\begin{gather*}
Ap(T_{b'}(n),s_0) = \{\sum_{i=1}^{n+1} t_i s_i | (t_1,\cdots,t_{n+1}) \in R_{b'}(n)\}.
\end{gather*}
\end{theorem}
\begin{corollary}
In a similar way, we obtain the explicit form of the Apery set of Thabit numerical semigroups of the second kind base $b$ for $n = 0$ and $n = 1$ and we obtain the genus of Thabit numerical semigroups of the second kind base $b$ for $n = 0$ and $n = 1$.
\begin{enumerate}
\item If $n = 0$ and $b \not\equiv 1\pmod{3}$, $Ap(T_{b'}(0),s_0) = \{t_1 s_1 | t_1 \in \{0,1,\cdots,b+1\}\} = \{b+2, b^2 + b + 1, \cdots, b^3 + 2b^2 + 2b + 1\}$ since $s_1 \equiv -(b-1) \pmod{b+2}$ and $\{0,-(b-1),\cdots,-(b+1)(b-1)\}$ is a complete system of residues modulo $b+2$. Hence, we obtain 
\begin{align*}
\sum_{(t_1) \in R_{b'}(0) } t_1 s_1 & = \sum_{k=1}^{b+1} ks_1 \\
& = (b+2)\{\frac{1}{2}(b+1)(b^2 + b + 1)\}
\end{align*}
and $g(T_{b'}(0)) = \frac{1}{2}(b+1)(b^2 + b + 1) - \frac{b+1}{2} = \frac{b^3 + 2b^2 + b}{2}$
\item If $n = 1$ and $b \not\equiv 1\pmod{3}$, $Ap(T_{b'}(1),s_0) = \{t_1 s_1 + t_2 s_2 | t_1 \in \{0,1,\cdots,b\}, t_2 \in \{0,1,\cdots,b-1\}\} \bigcup \{bs_2\} = \{0,s_1,\cdots, bs_1, s_2, s_1+s_2, \cdots, bs_1 + s_2, \cdots, (b-1)s_2, s_1 + (b-1)s_2, \cdots, bs_1 + (b-1)s_2, bs_2\}$ since $s_1 \equiv -(b-1) \pmod{b^2 + b + 1}$ and $\{0,-(b-1),\cdots,-(b^2+b)(b-1)\}$ is a complete system of residues modulo $b^2 + b + 1$. Hence, we obtain 
\begin{align*}
\sum_{(t_1,t_2) \in R_{b'}(1) } (t_1 s_1 + t_2 s_2) & = b\sum_{k=1}^{b} ks_1 + (b+1)\sum_{k=1}^{b-1} ks_2 + bs_2 \\
& = \frac{1}{2} (b^2 + b + 1)(b^5 + b^4 + b^3 + b)
\end{align*}
and $g(T_{b'}(1)) = \frac{1}{2}(b^5 + b^4 + b^3 + b) - \frac{b^2 + b}{2} = \frac{b^5 + b^4 + b^3 - b^2}{2}$
\end{enumerate}
\end{corollary}
\begin{corollary}
We obtain the maximal element in the Apery set of Thabit numerical semigroup of the second kind base $b$ and the Frobenius number of this semigroup is obtained immediately as follows:
\begin{enumerate}
\item If $n = 0$, $Ap(T_{b'}(0),s_0) = \{t_1 s_1 | t_1 \in \{0,1,\cdots,b+1\}\} = \{b+2, b^2 + b + 1, \cdots, b^3 + 2b^2 + 2b + 1\}$ implies that $max(Ap(T_{b'}(0),s_0)) = (b+1)s_1 = b^3 + 2b^2 + 2b + 1$ and $F(T_{b'}(0)) = (b+1)s_1 - s_0 = b^3 + 2b^2 + b - 1$. 
\item If $n = 1$, $Ap(T_{b'}(1),s_0) = \{t_1 s_1 + t_2 s_2 | t_1 \in \{0,1,\cdots,b\}, t_2 \in \{0,1,\cdots,b-1\}\} \bigcup \{bs_2\}$ implies that $max(Ap(T_{b'}(1),s_0)) = bs_1 + (b-1)s_2 = b^5 + b^4 + 2b - 1$ and $F(T_{b'}(1)) = bs_1 + (b-1)s_2 - s_0 = b^5 + b^4 - b^2 + b - 2$.
\item If $n \geq 2$, $Ap(T_{b'}(n),s_0) = (\{\sum_{i=1}^{n+1} t_i s_i | (t_1,\cdots,t_{n+1}) \in R_{b'}(n)\}$ implies that $max(Ap(T_{b'}(n),s_0) = 2s_1 + (b-1)s_n + (b-1)s_{n+1}$ and $F(T_{b'}(n)) = 2s_1 + (b-1)s_n + (b-1)s_{n+1} - s_0 = b^{2n+3} + b^{2n+2} - b^{2n+1} - b^{2n} + 2b^{n+2} + 2b^{n+1} + 2b^2$. 
\end{enumerate} 
\end{corollary}
Finally, we obtain the genus of Thabit numerical semigroups of the second kind base $b$ for $n \geq 2$.
\begin{theorem}
Let $n,b \in \mathbb{N}$ and $n,b \geq 2$. Then
\begin{gather*}
g(T_{b'}(n)) = 3b + \frac{b^{2n}(b^3 + b^2 - b - 1) + b^n\{b^2(n+1) - (n+3)\})}{2}.
\end{gather*}
\end{theorem}
\section{The Results of the semigroups $\big< \{b^{n+i} + 1 | i \in \mathbb{N}\}\big>$ for $2 | b$}
\begin{theorem}\label{Thm_minimal_3}
If $n,b \in \mathbb{N}$, $2 | b$ and $n \neq 0$, then $\big<\{b^{n+i} + 1| i \in \{0,1,\cdots,n\}\}\big>$ is a minimal system of generators.
\end{theorem}
\begin{corollary}
Let $n,b \in \mathbb{N}$, $2 | b$ and let $SC^{+}(b,n)$ be a Cunningham numerical semigroup associated with $n$ and $b$. Then $e(SC^{+}(b,n)) = n + 1$.
\end{corollary}
We define $RC_{b}(n)$ for $b,n \geq 2$ and $2 | b$, as follows:
\begin{definition}
Let $b,n \geq 2$ and $2 | b$. Then $RC_{b}(n) = \{(t_1,t_2,\cdots,t_{n}) | t_i \in \{0,1,\cdots,b\}\}$ is defined by
\begin{enumerate}
\item If $t_i = b$, $t_j = 0$ for all $1 \leq j < i$.
\item $t_{n} \leq b-1$.
\item If $t_{n} = b-1$, then $t_1 \leq 1$ and all $t_i = 0$ for $i \neq 1,n$.
\end{enumerate}
\end{definition}
Then we obtain the following theorem:
\begin{theorem} 
Let  $b,n \geq 2$ and $2 | b$. Then we obtain
\begin{gather*}
Ap(SC^{+}(b,n),s_0) = \{\sum_{i=1}^{n} t_i s_i | (t_1,\cdots,t_{n}) \in RC_{b}(n)\}.
\end{gather*}
\end{theorem}
\begin{corollary}
In a similar way, we obtain the explicit form of the Apery set of Cunningham numerical semigroups for $n = 0$ and $n = 1$.
\begin{enumerate}
\item If $n = 0$ and $2 | b$, $Ap(SC^{+}(b,0),s_0) = \{t_1 s_1 | t_1 \in \{0,1\}\} = \{2, b+1\}$ since $s_1 \equiv 1 \pmod{2}$ and $\{0,1\}$ is a complete system of residues modulo $2$. Hence, we obtain 
\begin{gather*}
\sum_{(t_1) \in RC_{b}(0) } t_1 s_1 = s_1 = b + 1
\end{gather*}
and $g(SC^{+}(b,0)) = \frac{1}{2}(b+1) - \frac{1}{2} = \frac{b}{2}$.
\item If $n = 1$ and $2 | b$, $Ap(SC^{+}(b,1),s_0) = \{t_1 s_1 | t_1 \in \{0,1,\cdots,b\} = \{0,s_1,\cdots, bs_1\} = \{0, b^2 + 1, \cdots, b^3 + b\}$ since $s_1 \equiv -(b-1) \pmod{b+1}$ and $\{0,-(b-1),\cdots,-b(b-1)\}$ is a complete system of residues modulo $b+1$. Hence, we obtain 
\begin{align*}
\sum_{(t_1) \in RC_{b}(1) } t_1 s_1 & = \sum_{k=1}^{b} ks_1 \\
& = (b+1)\frac{1}{2}b(b^2 + 1)
\end{align*}
and $g(C^{+}(b,1)) = \frac{1}{2}b(b^2 + 1) - \frac{b}{2} = \frac{b^3}{2}$.
\end{enumerate}
\end{corollary}
\begin{corollary}
We obtain the maximal element in the Apery set of Cunningham numerical semigroups and the Frobenius number of this semigroup is obtained immediately as follows:
\begin{enumerate}
\item If $n = 0$, $max(Ap(SC^{+}(b,0),s_0)) = b+1$ and $F(SC^{+}(b,0)) = (b+1) - s_0 = b - 1$. 
\item If $n = 1$, $max(Ap(SC^{+}(b,1),s_0)) = b^3 + b$ and $F(SC^{+}(b,1)) = bs_1 - s_0 = b^3 - 1$.
\item If $n \geq 2$, $Ap(SC^{+}(b,n),s_0) = (\{\sum_{i=1}^{n} t_i s_i | (t_1,\cdots,t_{n}) \in RC_{b}(n)\}$ implies that 
\begin{align*}
max(Ap(SC^{+}(b,n),s_0) & = s_1 + (b-1)s_{n} \text{ and } \\
F(SC^{+}(b,n)) & = s_1 + (b-1)s_n - s_0 = b^{n+1} + 1 + (b-1)(b^{2n} + 1) - (b^n + 1) \\
& = (b-1)(b^{2n} + b^n + 1).
\end{align*} 
\end{enumerate} 
\end{corollary}
Finally, we obtain the genus of Cunningham numerical semigroups for $n \geq 2$.
\begin{theorem}
Let $n,b \in \mathbb{N}$ and $n,b \geq 2$ where $2 | b$. Then
\begin{gather*}
g(SC^{+}(b,n)) = b + \frac{b^{2n}(b-1) + b^n (bn - n -1)}{2}.
\end{gather*}
\end{theorem}
\section{The Results of the semigroups $\big< \{b^{b^{n+i}} + 1 | i \in \mathbb{N}\}\big>$ for $2 | b$}
\begin{theorem} \label{thm_fer_main}
Let $n,b \in \mathbb{N}$ and $2 | b$, then we have $SF(b,n) = \big< b^{b^{n+i}} + 1 | i \in \mathbb{N}\big> = \big< b^{b^{n}} + 1, b^{b^{n+1}} + 1\big>$.
\end{theorem}
\begin{corollary}
Let $s_i = b^{b^{n+i}} + 1$. Then theorem \ref{thm_fer_main} implies that $\text{Ap}(SF(b,n), s_0) = \{0,s_1,2s_1,\cdots,(b^{b^{n}})s_1\}$. Hence $F(SF(b,n)) = b^{b^{n}} (b^{b^{n+1}} + 1) - (b^{b^{n}} + 1) = b^{(b+1)b^{n}} - 1$.  Also, $g(SF(b,n)) = \frac{(b^{b^{n}})(b^{b^{n}} + 1) \cdot (b^{b^{n+1}} + 1)}{2(b^{b^{n}} + 1)} - \frac{b^{b^{n}}}{2} = \frac{1}{2}b^{b^{n}} (b^{b^{n}} - 2) = \frac{1}{2} b^{b^{n}} b^{b^{n+1}} = \frac{1}{2} b^{(b+1)b^{n}}$.
\end{corollary}
\begin{example}
Let us consider the case of Fermat number. We know that $\big< \{2^{2^{n+i}} + 1 | i \in \mathbb{N}\}\big> = \big< \{2^{2^{n}} + 1, 2^{2^{n+1}} + 1\}\big>$ and it implies that $\text{Ap}(SF(2,n), s_0) = \{0,s_1,2s_1,\cdots,(2^{2^{n}})s_1\}$. Hence $F(SF(2,n)) = 2^{3\cdot 2^{n}} - 1$ and $g(SF(2,n)) = 2^{3\cdot 2^{n} - 1}$.
\end{example}

\end{document}